\documentclass[12pt]{amsart}
\usepackage[headings]{fullpage}
\usepackage[colorlinks,citecolor=blue,urlcolor=black,linkcolor=black]{hyperref}
\usepackage{amssymb,amsthm,url,cmap,stmaryrd}

\newtheorem{theorem}{Theorem}[section]
\newtheorem{cor}[theorem]{Corollary}
\newtheorem{prop}[theorem]{Proposition}
\newtheorem{fact}[theorem]{Fact}
\newtheorem{lemma}[theorem]{Lemma}
\newtheorem{claim}{Claim}[theorem]
\newtheorem*{Cor}{Corollary}
\newtheorem*{thma}{Theorem~A}
\newtheorem*{thmb}{Theorem~B}

\theoremstyle{definition}
\newtheorem{definition}[theorem]{Definition}
\newtheorem{q}[theorem]{Question}

\theoremstyle{remark}
\newtheorem{remark}[theorem]{Remark}
\newtheorem{remarks}[theorem]{Remarks}
\newtheorem*{uremark}{Remark}

\DeclareMathOperator{\cf}{cf}
\DeclareMathOperator{\im}{Im}
\DeclareMathOperator{\ns}{NS}
\DeclareMathOperator{\Tr}{Tr}
\DeclareMathOperator{\acc}{acc}

\DeclareMathOperator{\ord}{ORD}
\DeclareMathOperator{\otp}{otp}
\DeclareMathOperator{\reg}{Reg}
\DeclareMathOperator{\snr}{SNR}
\DeclareMathOperator{\prt}{\Pi}
\DeclareMathOperator{\prp}{\coprod}

\newcommand\s{\subseteq}

\newcommand\zfc{\textsf{\textup{ZFC}}}
\newcommand\diagonal{\bigtriangleup}
\newcommand\symdiff{\mathbin\triangle}
\renewcommand\mid{\mathrel{|}\allowbreak}

\title{Partitioning a reflecting stationary set}
\author{Maxwell Levine}
\address{Universit\"at Wien, Kurt G\"odel Research Center for Mathematical Logic, Austria}
\urladdr{http://www.logic.univie.ac.at/\textasciitilde levinem85}

\author{Assaf Rinot}
\address{Department of Mathematics, Bar-Ilan University, Ramat-Gan 5290002, Israel.}
\urladdr{http://www.assafrinot.com}

\thanks{The second author was partially supported by the European Research Council (grant agreement ERC-2018-StG 802756) and by the Israel Science Foundation (grant agreement 2066/18).}


\begin{document}
\begin{abstract} We address the question of whether a reflecting stationary set may be partitioned into two or more reflecting stationary subsets,
providing various affirmative answers in \zfc.
As an application to singular cardinals combinatorics, we infer that it is never the case that there exists a singular cardinal all of whose scales are very good.
\end{abstract}
\date{\today}
\maketitle
\section{Introduction}
A fundamental fact of set theory is \emph{Solovay's partition theorem} \cite{MR0290961} asserting that, for every stationary subset $S$ of a regular uncountable cardinal $\kappa$,
there exists a partition $\langle S_i\mid i<\kappa\rangle$ of $S$ into stationary sets.
The standard proof involves the analysis of a certain $C$-sequence over a stationary subset of $S$;
such a sequence exists in $\zfc$ (cf.~\cite{MR2519245}), but assuming the existence of better $C$-sequences, stronger partition theorems follow. For instance:
\begin{itemize}
\item (folklore) If $\kappa=\lambda^+$ and $\square_\lambda$ holds, then any stationary subset $S$ of $\kappa$ may be partitioned into non-reflecting stationary sets $\langle S_i\mid i<\kappa\rangle$.
That is, for all $i<\kappa$,  $S_i$ is stationary, but $\Tr(S_i):=\{\beta<\sup(S_i)\mid \cf(\beta)>\omega\ \&\ S_i\cap\beta\text{ is stationary in }\beta\}$ is empty.
\item \cite[Lemma~3.2]{paper18} If  $\square(\kappa)$ holds,
then for every stationary $S\s\kappa$, there exists a coherent $C$-sequence $\vec C=\langle C_\alpha\mid\alpha<\kappa\rangle$ such that $S_i:=\{ \alpha\in S\mid \min(C_\alpha)=i\}$ is stationary for all $i<\kappa$.
The coherence of $\vec C$ implies that the elements of $\langle \Tr(S_i)\mid i<\kappa\rangle$ are pairwise disjoint as well.
\item \cite[Theorem~1.24]{paper29} If $\square(\kappa)$ holds, then any fat subset of $\kappa$ may be partitioned into $\kappa$-many fat sets
that do not simultaneously reflect.
\end{itemize}
This raises the question of whether there is a fundamentally different way to partition large sets. A more concrete question reads as follows:

\begin{q}\label{q1} Suppose that $S$ is a subset of a regular uncountable cardinal $\kappa$ for which $\Tr(S)$ is stationary.
Can $S$ be split into sets $\langle S_i\mid i<\theta\rangle$ in such a way that $\bigcap_{i<\theta}\Tr(S_i)$ is stationary? And how large can $\theta$ be?
\end{q}
\begin{uremark}
Note that for any sequence $\langle S_i\mid i<\theta\rangle$ of pairwise disjoint subsets of $\kappa$, the intersection $\bigcap_{i<\theta}\Tr(S_i)$ is a subset of $E^\kappa_{\ge\theta}$.
Therefore, the only cardinals $\theta$ of interest are the ones for which $\kappa\setminus\theta$ still contains a regular cardinal.
\end{uremark}

The above question leads us to the following principle:
\begin{definition}\label{principle}  $\prt(S,\theta,T)$ asserts the existence of a partition $\langle S_i\mid i<\theta\rangle$ of $S$ such that  $T\cap\bigcap_{i<\theta}\Tr(S_i)$ is stationary.
\end{definition}

In Magidor's model \cite[\S2]{magidor}, $\prt(S,\aleph_1,T)$ holds for any two stationary subsets $S\s E^{\aleph_2}_{\aleph_0}$ and $T\s E^{\aleph_2}_{\aleph_1}$,
and it is also easy to provide consistent affirmative answers to Question~\ref{q1} without appealing to large cardinals.
However, the focus of this short paper is to establish that instances of the principle $\prt(\ldots)$ hold true in $\zfc$.
It is proved:

\begin{thma} Suppose that $\mu<\theta<\lambda$ are infinite cardinals, with $\mu,\theta$ regular.
\begin{enumerate}
\item If $\lambda$ is inaccessible, then $\prt(\lambda,\theta,\lambda)$  and $\prt(\lambda^+,\lambda,\lambda^+)$ both hold;
\item If $\lambda$ is regular, then $\prt(E^{\lambda^+}_\mu,\theta,E^{\lambda^+}_\theta)$ holds;
\item If $2^\theta\le\lambda$ and $\theta\neq\cf(\lambda)$, then $\prt(E^{\lambda^+}_\mu,\theta,E^{\lambda^+}_\theta)$ holds;
\item If $\lambda$ is singular, then $\prt(E^{\lambda^+}_\mu,\theta,E^{\lambda^+}_{\le\theta^{+3}})$ holds.
\end{enumerate}
\end{thma}

It is worth mentioning that our proof of Clause~(4) is indeed fundamentally different than all standard proofs for partitioning a stationary set.
We build on the fact that any singular cardinal admits a scale and that
the set of good points of a scale is stationary relative to any cofinality;
we also use a combination of club-guessing with Ulam matrices to avoid any cardinal arithmetic hypotheses.

\medskip

We initiated this project since we realized that $\zfc$ instances of $\prt(\ldots)$ would allow us to prove that the statement ``all scales are very good'' is inconsistent.
And, indeed, the following is an easy consequence of Theorem~A:
\begin{Cor} Suppose that $\lambda$ is a singular cardinal, and $\vec\lambda$ is a sequence of a regular cardinals of length $\cf(\lambda)$, converging to $\lambda$. If $\prod\vec\lambda$ carries a scale, then it also carries a scale which is not very good.
\end{Cor}

In this short paper, we also consider a simultaneous version of the principle $\prt(\ldots)$ which is motivated by a simultaneous version of Solovay's partition theorem recently obtained by Brodsky and Rinot:
\begin{lemma}[{\cite[Lemma~1.15]{paper29}}]\label{lemma13} Suppose that
$\langle S_i\mid i<\theta\rangle$ is a sequence of stationary subsets of a regular uncountable cardinal $\kappa$, with $\theta\le\kappa$.
Then there exists a sequence $\langle S_i'\mid i\in I\rangle$ of pairwise disjoint stationary sets such that:
\begin{itemize}
\item  $S_i'\s S_i$ for every $i\in I$;
\item $I$ is a cofinal subset of $\theta$.\footnote{Note that if we demand that $I$ be equal to $\theta$, then we do not get a theorem in $\zfc$.
For instance, if the nonstationary ideal on $\omega_1$ is $\omega_1$-dense \cite[Theorem~6.148]{MR2723878}, then we could
let $\langle S_i\mid i<\omega_1\rangle$ be a non-injective enumeration of a dense subset of $\ns_{\omega_1}$.} 
\end{itemize}
\end{lemma}

Evidently, Solovay's theorem follows  by invoking the preceding theorem with a constant sequence of length $\theta=\kappa$.

The simultaneous version of Definition~\ref{principle} reads as follows.
\begin{definition}\label{principle2}  $\prp(S,\nu,T)$ asserts
that for every $\theta\le\nu$, every sequence $\langle S_i\mid i<\theta\rangle$ of subsets of $S$
and every stationary $T'\s T\cap\bigcap_{i<\theta}\Tr(S_i)$,
there exists a sequence $\langle S_i'\mid i\in I\rangle$ of pairwise disjoint stationary sets such that:
\begin{itemize}
\item  $S_i'\s S_i$ for every $i\in I$;
\item $I$ is a cofinal subset of $\theta$;
\item $T'\cap \bigcap_{i\in I}\Tr(S'_i)$ is stationary.
\end{itemize}
\end{definition}

It is proved:

\begin{thmb} Suppose that $\nu<\lambda$ are uncountable cardinals with $\nu\neq\cf(\lambda)$, and $2^\nu\le\lambda$.
Then, any of the following implies that $\prp(\lambda^+,\nu,E^{\lambda^+}_\nu)$ holds:
\begin{enumerate}
\item $\lambda$ is a regular cardinal;
\item $\lambda$ is a singular cardinal admitting a very good scale.
\end{enumerate}
\end{thmb}

\subsection*{Notation and conventions} For cardinals $\theta<\kappa$, we let $E^\kappa_\theta:=\{\alpha<\kappa\mid\cf(\alpha)=\theta\}$; $E^\kappa_{\neq\theta}$, $E^\kappa_{>\theta}$, $E^\kappa_{\ge\theta}$ and $E^\kappa_{\le\theta}$ are defined similarly.
For a set of ordinals $a$, we write $\acc^+(a) := \{\alpha < \sup(a) \mid \sup(a \cap \alpha) = \alpha > 0\}$ and $\acc(a) := a \cap \acc^+(a)$.
The class of all ordinals is denoted by $\ord$. We also let $\reg(\kappa):=\{\theta<\kappa\mid \aleph_0\le\cf(\theta)=\theta\}$.

\section{\textit{pcf} scales}
In this section, we recall the notion of a scale for a singular cardinal (also known as a \textit{pcf scale}) and the classification of points of a scale.
These concepts will play a role in the proof of Theorems A and B.
In turn, we also present an application of the partition principle $\prt(\ldots)$ to the study of very good scales.

\begin{definition} Suppose that $\lambda$ is a singular cardinal, and $\vec\lambda=\langle\lambda_i\mid i<\cf(\lambda)\rangle$ is a strictly increasing sequence of regular cardinals, converging to $\lambda$.
For any two functions $f,g\in\prod\vec\lambda$ and $i<\cf(\lambda)$, we write $f<^ig$ iff, for all $j\in\cf(\lambda)\setminus i$, $f(j)<g(j)$.
We also write $f<^* g$ to expresses that $f<^i g$ for some $i<\cf(\lambda)$.
\end{definition}
\begin{definition}\label{def22} Suppose that $\lambda$ is a singular cardinal.
 A sequence $\vec f=\langle f_\gamma\mid\gamma<\lambda^+\rangle$ is said to be a \emph{scale for $\lambda$} iff there exists a sequence $\vec \lambda$ as in the previous definition, such that:
\begin{itemize}
\item for every $\gamma<\lambda^+$, $f_\gamma\in\prod\vec\lambda$;
\item for every $\gamma<\delta<\lambda^+$, $f_\gamma<^*f_\delta$;
\item for every $g\in\prod\vec\lambda$, there exists $\gamma<\lambda^+$ such that $g<^*f_\gamma$.
\end{itemize}

 An ordinal $\alpha<\lambda^+$ is said to be:
\begin{itemize}
\item \emph{good} with respect to $\vec f$ iff there exist a cofinal subset $A\s\alpha$ and $i<\cf(\lambda)$ such that, for every pair of ordinals $\gamma<\delta$ from $A$, $f_\gamma<^if_\delta$;
\item \emph{very good} with respect to $\vec f$ iff there exist a club $C\s\alpha$ and $i<\cf(\lambda)$ such that, for every pair of ordinals $\gamma<\delta$ from $C$, $f_\gamma<^if_\delta$.
\end{itemize}

We also let:
\begin{itemize}
\item $G(\vec f):=\{\alpha\in E^{\lambda^+}_{\neq\cf(\lambda)}\mid \alpha\text{ is good with respect to }\vec f\}$;
\item $V(\vec f):=\{\alpha\in E^{\lambda^+}_{\neq\cf(\lambda)}\mid \alpha\text{ is very good with respect to }\vec f\}$.
\end{itemize}
\end{definition}

Clearly, $E^{\lambda^+}_{<\cf(\lambda)}\s V(\vec f)\s G(\vec f)$. It is also not hard to see that if $\vec f,\vec g$ are two scales in the same product $\prod\vec\lambda$, then they interleave each other on a club,
so that $G(\vec f)\symdiff G(\vec g)$ is nonstationary.
This means that, up to a club, the set of good points is in fact an invariant of $\vec\lambda$.
We shall soon discuss the question of whether the same is true for the set of very good points, but let us first recall a few fundamental results of Shelah.\footnote{For an excellent survey, see \cite{MR2768694}.}

\begin{fact}[Shelah, \cite{Sh:420},\cite{Sh:g}]\label{scales} For every singular cardinal $\lambda$:
\begin{enumerate}
\item\label{c1}  There exists a scale for $\lambda$;
\item\label{c2} For any scale $\vec f$ for $\lambda$ and every $\theta\in\reg(\lambda)\setminus\{\cf(\lambda)\}$,
the intersection $G(\vec f)\cap E^{\lambda^+}_\theta$ is stationary;
\item\label{c3} If $\vec f=\langle f_\gamma\mid\gamma<\lambda^+\rangle$ is a scale for $\lambda$ in a product $\prod_{i<\cf(\lambda)}\lambda_i$ and $\alpha\in E^{\lambda^+}_{\neq\cf(\lambda)}$, then $\alpha\in G(\vec f)$ iff there exists $e\in \prod_{i<\cf(\lambda)}\lambda_i$
satisfying $\cf(e(i))=\cf(\alpha)$ whenever $\lambda_i>\cf(\alpha)$,
and such that $e$ forms an \emph{exact upper bound} for $\vec f\restriction\alpha$, i.e.:
\begin{itemize}
\item for all $\gamma<\alpha$, $f_\gamma<^*e$;
\item for all $g\in\prod_{i<\cf(\lambda)}\lambda_i$ with $g<^*e$, there is $\gamma<\alpha$ with $g<^* f_\gamma$.
\end{itemize}
\end{enumerate}
\end{fact}
\begin{definition}
A scale $\vec f$ for $\lambda$ is said to be \emph{good} (resp. \emph{very good}) iff there exists a club $D\s\lambda^+$ such that $D\cap E^{\lambda^+}_{\neq\cf(\lambda)}\s G(\vec f)$
(resp. $D\cap E^{\lambda^+}_{\neq\cf(\lambda)}\s V(\vec f)$).
\end{definition}

Our Definition~\ref{principle} is motivated by a proof of a result of Cummings and Foreman \cite[Theorem~3.1]{cummings2010} asserting that if $V=L$, then
 $\prod_{n<\omega}\aleph_n$ carries a very good scale and yet another scale which fails to be very good at every point of uncountable cofinality.
Among other things, their proof shows:
\begin{prop}\label{prop25} Suppose that $\vec\lambda=\langle\lambda_i\mid i<\cf(\lambda)\rangle$ is a strictly increasing sequence of cardinals, converging to a singular cardinal $\lambda$, and $\prod\vec\lambda$ carries a scale.
If $\prt(\lambda^+,\cf(\lambda),E^{\lambda^+}_{>\cf(\lambda)})$ holds, then $\prod\vec\lambda$ also carries a scale which is not very good.
\end{prop}
\begin{proof} We repeat the argument from \cite[\S3]{cummings2010}.
Let $\vec f = \langle f_\gamma\mid \gamma<\lambda^+\rangle$ be a scale in $\prod\vec\lambda$.
By $\prt(\lambda^+,\cf(\lambda),E^{\lambda^+}_{>\cf(\lambda)})$, we fix a partition $\langle S_i\mid i<\cf(\lambda)\rangle$ of $\lambda^+$
for which $T:=E^{\lambda^+}_{>\cf(\lambda)}\cap \bigcap_{i<\cf(\lambda)}\Tr(S_i)$ is stationary.
Now, define a new scale $\vec g = \langle g_\gamma\mid \gamma<\lambda^+\rangle$ by letting, for all $\gamma<\lambda^+$ and $i<\cf(\lambda)$,
$$g_\gamma(i):=\begin{cases}0,&\text{if }\gamma\in S_i;\\
f_\gamma(i),&\text{otherwise}.\end{cases}$$

Evidently, $f_\gamma$ and $g_\gamma$ differ on at most a single index, and so $\vec g$ is a scale. However, $\vec g$ fails to be very good at any given point $\alpha \in T$.
To see this, fix an arbitrary club $C \s \alpha$ and an index $i<\cf(\lambda)$. Let $\gamma:=\min(C\cap S_i)$ and ${\delta}:=\min(C\cap S_i\setminus(\gamma+1))$. Then $\gamma<{\delta}$ is a pair of elements of $C$, while $g_\gamma(i) = 0 = g_{\delta}(i)$.
\end{proof}

\begin{remark}
Gitik and Sharon \cite{Gitik-Sharon2008} constructed a model in which $\aleph_{\omega^2}$ carries a very good scale in one product and a bad (i.e., not good) scale in another --- hence $\square_{\aleph_{\omega^2}}$, let alone $V=L$, cannot hold.
The question, then, is whether the former product carries only very good scales. Our results show that it does not, and in fact that it is a theorem of $\zfc$ that in any product that carries a scale, there are scales which are not very good.
Furthermore, it follows from the preceding proof together with Corollary~\ref{cor34} below (using $\theta:=\cf(\lambda)$) that any scale for a singular cardinal $\lambda$ may be manipulated to have its set of very good points to not be a club relative to cofinality $\nu$ for unboundedly many regular cardinals $\nu<\lambda$.
\end{remark}

\section{Theorem~A}
Recall that $\mathcal D(\nu,\theta)$ stands for $\cf([\nu]^\theta,\supseteq)$,\footnote{Take note of the direction of the containment.} i.e., the least size of a family $\mathcal D\s[\nu]^\theta$
with the property that for every  $a\in[\nu]^\theta$, there is $d\in\mathcal D$ with $d\s a$.
Suppose now that $\nu$ is regular and uncountable;
we let $\mathcal C(\nu,\theta)$ denote the least size of a family $\mathcal C\s[\nu]^\theta$
with the property that for every club $b$ in $\nu$, there is $c\in\mathcal C$ with $c\s b$.
It is well-known that $\mathcal C(\nu,\nu)$, better known as $\cf(\ns_\nu,\s)$, can be arbitrarily large. In contrast, for small values of $\theta$, $\mathcal C(\nu,\theta)$ is provably small:

\begin{lemma}\label{lemmaC} Suppose that $\theta\le\nu$ are cardinals, with $\nu$ regular uncountable. Then:
\begin{itemize}
\item $\nu\le\mathcal C(\nu,\theta)\le\mathcal D(\nu,\cf(\theta))\le 2^\nu$;
\item $\mathcal C(\nu,\theta)=\nu$ whenever $\theta^{++}<\nu$;
\item $\mathcal C(\nu,\theta)=\nu$ whenever $\theta<\cf(\theta)^+<\nu$.
\end{itemize}
\end{lemma}
\begin{proof}
Evidently, $\mathcal D(\nu,\cf(\theta))\le\nu^{\cf(\theta)}\le\nu^\nu=2^\nu$.
As, for all $\alpha<\nu$, $\nu\setminus\alpha$ is a club in $\nu$, we also have $\nu\le\mathcal C(\nu,\theta)$.
In addition, it is obvious that if $\cf(\theta)=\theta$ then  $\mathcal C(\nu,\theta)\le\mathcal D(\nu,\cf(\theta))$.

Next, suppose that $\mu$ is an arbitrary infinite regular cardinal, with $\mu^+<\nu$. By Shelah's club-guessing theorem \cite[III.\S2]{Sh:g},
there exists a sequence $\langle c_\alpha\mid \alpha\in E^\nu_\mu\rangle$ having the following crucial property: for every club $c$ in $\nu$, there exists $\alpha\in E^\nu_\mu$ such that $c_\alpha\s c\cap\alpha$ and  $\otp(c_\alpha)=\mu$.
It follows that $\{ c_\alpha\mid \alpha\in E^\nu_\mu\}$ witnesses that $\mathcal C(\nu,\mu)\le\nu$.
In particular:
\begin{itemize}
\item if $\theta^{++}<\nu$, then using $\mu:=\theta^+$. we have $\mathcal C(\nu,\theta)\le\mathcal C(\nu,\mu)=\nu$;
\item if $\theta<\cf(\theta)^+<\nu$, then using $\mu:=\theta$, we have $\mathcal C(\nu,\theta)=\mathcal C(\nu,\mu)=\nu$.
\end{itemize}

Finally, we are left with dealing with the case that $\nu\in\{\theta^+,\theta^{++}\}$ and $\cf(\theta)<\theta$.
For every $\mu\in\reg(\theta)$, fix a family $\mathcal C_\mu$ witnessing that $\mathcal C(\nu,\mu)=\nu$.
Fix an enumeration $\{ c_\delta\mid \delta<\nu\}$ of $\bigcup_{\mu\in\reg(\theta)}\mathcal C_\mu$.
Also, fix a family $\mathcal D$ witnessing the value of $\mathcal D(\nu,\cf(\theta))$.
Then, let $\mathcal C:=\{\bigcup_{\delta\in d} c_\delta\mid d\in\mathcal D\}$.
Evidently, $\mathcal C(\nu,\theta)\le|\mathcal C|\le|\mathcal D|=\mathcal D(\nu,\cf(\theta))$.
\end{proof}

\begin{cor} For every infinite cardinal $\theta$ and every cardinal $\lambda\ge\mathcal D(\theta,\cf(\theta))$, we have $\mathcal C(\nu,\theta)\le\lambda$ whenever $\nu\in\reg(\lambda)\setminus\theta$.
\end{cor}
\begin{proof}
First, note that $\mathcal D(\theta^{+n},\cf(\theta))\le\max\{\theta^{+n},\mathcal D(\theta,\cf(\theta))\}$ for all $n<\omega$.

We now prove the contrapositive. Suppose that $\nu$ is a regular cardinal and $\mathcal C(\nu,\theta)>\lambda>\nu\ge\theta$.
Then, by Lemma~\ref{lemmaC}, $\nu=\theta^{+n}$ for some $n<3$ and $\mathcal D(\nu,{\cf(\theta)})>\lambda$.
It follows that $$\lambda<\mathcal D(\nu,\cf(\theta))\le\max\{\nu,\mathcal D(\theta,\cf(\theta))\}\le\max\{\lambda,\mathcal D(\theta,\cf(\theta)\},$$
and hence $\mathcal D(\theta,{\cf(\theta)})>\lambda$.
\end{proof}
\begin{remark} See \cite{MR3563076} for a study of the map $\theta\mapsto\mathcal D(\theta,\cf(\theta))$ over the class of singular cardinals.
\end{remark}
A main aspect of the upcoming proofs is the analysis of local versus global features of a function.
For this, it is useful to establish the following lemma.

\begin{lemma}\label{localtoglobal}
For any ordinal $\zeta$ of uncountable cofinality, there exists a class map $\psi_\zeta:\ord\rightarrow\cf(\zeta)$ satisfying the following.
For every ordinal $\alpha$ with $\cf(\alpha)=\cf(\zeta)$, every function $f:\alpha\rightarrow\ord$
and every stationary $s\s\alpha$ such that $f\restriction s$ is strictly increasing and converging to $\zeta$,
there exists a club $c\s\alpha$ such that  $(\psi_\zeta\circ f)\restriction(c\cap s)$ is strictly increasing.

\end{lemma}
\begin{proof} Let $\nu$ be some regular uncountable cardinal, and let $\zeta$ be an ordinal of cofinality $\nu$.
For every ordinal $\alpha$ of cofinality $\nu$,
fix a strictly increasing function $\pi_\alpha:\nu\rightarrow\alpha$ whose image is a club in $\alpha$.
Define $\psi_\zeta:\ord\rightarrow\nu$ by letting:
$$\psi_\zeta(\eta):=\begin{cases}
\min\{ i<\nu\mid \eta<\pi_\zeta(i)\},&\text{if }\eta<\zeta;\\
0,&\text{otherwise}.
\end{cases}$$

Now, suppose that we are given a function $f:\alpha\rightarrow\ord$ along with a stationary $s\s \alpha$
on which $f$ is strictly increasing and converging to $\zeta$. Let $\bar f:=\psi_\zeta\circ f\circ \pi_\alpha$,
which is a function from $\nu$ to $\nu$.
Consider the club $\bar c:=\{\bar\delta<\nu\mid \bar f[\bar\delta]\s\bar\delta\}$,
and the set $\bar t:=\{\bar\delta<\nu\mid \pi_\alpha(\bar\delta)\in s\ \&\ \bar f(\bar\delta)<\bar\delta\}$.

If $\bar t$ is stationary, then by Fodor's lemma, we may fix some stationary $\hat t\s\bar t$ and some $i<\nu$ such that $\bar f[\hat t]=\{i\}$.
As $\pi_\alpha[\hat t]$ is a cofinal (indeed, stationary) subset of $s$, and $f\restriction s$ is strictly increasing and converging to $\zeta$, we may find a large enough $\delta\in\pi_\alpha[\hat t]$ such that $\eta:=f(\delta)$ is greater than $\pi_\zeta(i)$.
But then, for $\bar\delta:=\pi_\alpha^{-1}(\delta)$, we have $\bar f(\bar\delta)=\psi_\zeta(f(\pi_\alpha(\bar\delta)))=\psi_\zeta(f(\delta))=\psi_\zeta(\eta)>i$, contradicting the fact that $\bar\delta\in\hat t$.

So $\bar t$ is nonstationary, and hence we may find a club $c$ in $\alpha$ with $c\s\pi_\alpha[\bar c\setminus\bar t]$.
To see that $(\psi_\zeta\circ f)\restriction (c\cap s)$ is strictly increasing, let $\gamma<\delta$ be an arbitrary pair of elements of $c\cap s$. Put $\bar\gamma:=\pi_\alpha^{-1}(\gamma)$ and $\bar\delta:=\pi_\alpha^{-1}(\delta)$.
As $\bar\delta\in\bar c\setminus\bar t$ and $\pi_\alpha(\bar\delta)\in s$, we indeed have
\[(\psi_\zeta\circ f)(\gamma)=\bar f(\bar\gamma)<\bar\delta\le f(\bar\delta)=(\psi_\zeta\circ f)(\delta).\qedhere\]
\end{proof}

\begin{theorem}\label{thm35} Suppose:
\begin{itemize}
\item $\lambda$ is a singular cardinal;
\item $\vec f = \langle f_\beta\mid\beta<\lambda^+\rangle$ is a scale for $\lambda$;
\item  $\nu$ is a regular uncountable cardinal $\neq\cf(\lambda)$;
\item $S$ and $T$ are subsets of $\lambda^+$;
\item $G(\vec f)\cap \Tr(S)\cap E^{\lambda^+}_{\nu}\cap T$ is stationary.\footnote{Recall Definition~\ref{def22}.}
\end{itemize}
For every (finite or infinite) cardinal $\theta\le\nu$, if any of the following holds true:
\begin{enumerate}
\item[(i)] $\nu$ is an infinite successor cardinal and $\mathcal C(\nu,\theta)\le\lambda$;
\item[(ii)] $2^\nu\le\lambda$,
\end{enumerate}
then $\prt(S,\theta,T)$ holds.
\end{theorem}
\begin{proof} Let $\theta\le\nu$ be cardinal satisfying Clause (i) or (ii).
We shall prove that $\prt(S,\theta,T)$ holds by exhibiting a function $h:S\rightarrow\theta$ and some stationary $T'\s T$ such that $T'\s \Tr(h^{-1}\{\tau\})$ for all $\tau<\theta$.

Denote $T_0:=G(\vec f)\cap \Tr(S)\cap E^{\lambda^+}_{\nu}\cap T$.
For every $i<\cf(\lambda)$, denote $\lambda_i:=\sup\{ f_\beta(i)\mid \beta<\lambda^+\}$,
so that $\vec\lambda:=\langle\lambda_i\mid i<\cf(\lambda)\rangle$ is a strictly increasing sequence of regular cardinals, converging to $\lambda$.
Let $k<\cf(\lambda)$ be the least to satisfy $\lambda_k>\nu$.

\begin{claim}\label{claim351} Let $\alpha \in T_0$. There exist $i\in\cf(\lambda)\setminus k$ and $\varepsilon\in E^{\lambda_i}_\nu$ 
such that, for every $\gamma<\varepsilon$,
$\{\beta \in S \cap \alpha\mid \gamma\le f_\beta(i) <\varepsilon\}$ is stationary in $\alpha$.
\end{claim}
\begin{proof} If $\alpha\in V(\vec f)$,
then pick a club $C$ in $\alpha$ of order-type $\nu$
and a large enough $i\in\cf(\lambda)\setminus k$ such that $\langle f_\beta(i)\mid \beta\in C\rangle$ is strictly increasing.
Evidently, in this case, $i$ and $\varepsilon:=\sup_{\beta\in C}f_\beta(i)$ are as sought.
Next, suppose that $\alpha\notin V(\vec f)$, so that $\cf(\alpha)\ge\cf(\lambda)$.
Recalling that $\cf(\alpha)=\nu$ and $\nu\neq\cf(\lambda)$, this must mean that $\nu>\cf(\lambda)$.

Since $\alpha\in T_0\s G(\vec f)\cap E^{\lambda^+}_{>\cf(\lambda)}$,
we use Fact~\ref{scales}(\ref{c2}) to fix an exact upper bound $e_\alpha\in\prod\vec\lambda$ for $\vec f\restriction\alpha$ such that $\cf(e_\alpha(i))= \nu$ for all $i\in \cf(\lambda)\setminus k$.
Of course, we may also assume that $e_\alpha(i)>0$ for all $i<k$.
We will show that there exists $i\in\cf(\lambda)\setminus k$ for which $\varepsilon:=e_\alpha(i)$ satisfies the conclusion of the claim.

Define $g:\cf(\lambda)\rightarrow\lambda$ by letting, for all $i<k$, $g(i):=0$, and, for all $i\in\cf(\lambda)\setminus k$,
$$g(i):=\sup\{\gamma<e_\alpha(i)\mid \{\beta \in S \cap \alpha\mid \gamma\le f_\beta(i)< e_\alpha(i)\} \text{ is stationary in }\alpha\}+1.$$
Towards a contradiction, suppose that $g\in\prod_{i<\cf(\lambda)}e_\alpha(i)$.
For each $i\in \cf(\lambda)\setminus k$, pick a club $C_i$ in $\alpha$ such that, for all $\beta \in S\cap C_i$, $f_\beta(i)\notin [g(i),e_\alpha(i))$.
As $\cf(\alpha)>\cf(\lambda)$,  $C := \bigcap_{i\in\cf(\lambda)\setminus k}C_i$ is a club in $\alpha$.
By the choice of $e_\alpha$ and as $\cf(\alpha)>\cf(\lambda)$, we may also fix a stationary subset $B\s S\cap C$ and a large enough $j\in\cf(\lambda)\setminus k$
such that, for all $\beta\in B$, $f_\beta<^je_\alpha$.
It follows that, for all $\beta \in B$,  $f_\beta<^jg$.
But $B$ is cofinal in $\alpha$, so that, for all $\beta<\alpha$,  $f_\beta<^* g$.
This is a contradiction to the facts that $g\in\prod_{i<\cf(\lambda)}e_\alpha(i)$ and that $e_\alpha$ is an exact upper bound for $\vec f\restriction \alpha$.
\end{proof}

For each $\alpha\in T_0$, fix $i_\alpha$ and $\varepsilon_\alpha$ as in the claim.
Then fix some stationary $T_1 \s T_0$ along with $i^*<\cf(\lambda)$ and $\varepsilon\in E^{\lambda_{i^*}}_\nu$ such that $i_\alpha=i^*$ and $\varepsilon_\alpha=\varepsilon$ for all $\alpha\in T_1$.
Let $E$ be some club in $\varepsilon$ of order-type $\nu$.

\begin{claim}\label{claim212} Let $\alpha\in T_1$. Then at least one of the following holds true:\footnote{The first alternative is quite prevalent,
so that the second alternative is here for the rescue just in case that $\alpha$ is a good point which is not \emph{better} (see \cite[\S4]{MR2768694}) and $S\cap E^\alpha_{\neq\cf(\lambda)}$ is nonstationary.}
\begin{enumerate}
\item $\delta\mapsto f_\delta(i^*)$ is strictly increasing over some stationary subset of $S\cap\alpha$.
\item there is $D\in [E]^\nu$ such that, for any pair of ordinals $\gamma<\delta$ from $D$,
$\{\beta \in S \cap \alpha\mid \gamma< f_\beta(i^*)<\delta\}$ is stationary in $\alpha$.
\end{enumerate}
\end{claim}
\begin{proof} Suppose that Clause~(1) fails.
As $\cf(\varepsilon)=\nu$, to prove that Clause~(2) holds,
it suffices to show that, for all $\gamma\in E$, there is a large enough $\delta \in E$ such that
$\{\beta \in S \cap \alpha\mid \gamma<f_\beta(i^*) <\delta\}$ is stationary in $\alpha$.
Thus, let $\gamma\in E$ be arbitrary.

Fix a strictly increasing function $\pi_0:\nu\rightarrow\alpha$ whose image is a club in $\alpha$,
and a strictly increasing function $\pi_1:\nu\rightarrow \varepsilon$ whose image is $E$.
As $\alpha\in T_1$ and $i^*=i_\alpha$, we infer that
$$\bar S:=\pi_0^{-1}\{\beta \in S \cap \alpha\mid \gamma+1\le f_\beta(i^*)<\varepsilon\}$$ is stationary in $\nu$.
Define a function $\phi:\bar S\rightarrow\nu$ by stipulating:
$$\phi(\bar\beta):=\min\{\bar\delta<\nu\mid f_{\pi_0(\bar\beta)}(i^*)<\pi_1(\bar\delta)\}.$$
Let $\hat C:=\{\bar\beta\in\bar S\mid \phi[\bar\beta]\s\bar\beta\}$ and $\hat S:=\{ \bar\beta\in \bar S\mid \phi(\bar\beta)<\bar\beta\}$.

Note that if $\hat S$ is nonstationary, then $R:=\pi_0[\hat C\setminus \hat S]$ is a stationary subset of $S\cap\alpha$,
and for any pair of ordinals $\beta<\beta'$ from $R$, we have $$\phi(\pi_0^{-1}(\beta))<\pi_0^{-1}(\beta')\le \phi(\pi_0^{-1}(\beta')),$$
meaning that $f_\beta(i^*)<\pi_1(\phi(\pi_0^{-1}(\beta)))\le f_{\beta'}(i^*)$, and contradicting the fact that Clause~(1) fails.
So $\hat S$ must be stationary.

Fix a stationary subset $S'\s\hat S$ on which $\phi$ is constant, with value, say, $\bar\delta$.
Put $\delta:=\pi_1(\bar\delta)$, so that $\delta\in E$. Then $\pi_0[S']$ is a stationary subset of $S\cap\alpha$,
and, for all $\beta\in \pi_0[S']$, we have $\gamma<\gamma+1\le f_\beta(i^*)<\pi_1(\bar\delta)=\delta$, as sought.
\end{proof}

Let $T_2$ denote the set of all $\alpha\in T_1$ for which Clause~(2) of Claim~\ref{claim212} holds.

\medskip

\underline{Case 1.}
Suppose that $T_2$ is stationary.
For each $\alpha\in T_2$, fix some $D_\alpha$ as in the claim.
By replacing $D_\alpha$ with its closure, we may assume that $D_\alpha$ is a subclub of $E$.
As $E$ is the order-preserving continuous image of $\nu$ and as $\mathcal C(\nu,\theta)\le\lambda$,\footnote{Recall that by Lemma~\ref{lemmaC}, $\mathcal C(\nu,\theta)\le2^\nu$.}
we may fix some stationary $T_3 \s T_2$ and some  $D\s E$ of order-type $\theta$ such that $D\s D_\alpha$ for all $\alpha\in T_3$.
Define $h:S\rightarrow\theta$ by letting $h(\beta):=0$ whenever $f_\beta(i^*)\ge\sup(D)$, and $h(\beta):=\sup(\otp(f_\beta(i^*)\cap D))$, otherwise.
We claim that $T_3\s\Tr(h^{-1}\{\tau\})$ for all $\tau<\theta$.
To see this, let $\alpha\in T_3$ and $\tau<\theta$ be arbitrary. Let $\gamma$ denote the unique element of $D$ such that $\otp(D\cap\gamma)=\tau$.
Let $\delta:=\min(D\setminus(\gamma+1))$. As $\gamma<\delta$ is a pair of elements from $D_\alpha$, we know that $S':=\{ \beta\in S\cap\alpha\mid \gamma<f_\beta(i^*)<\delta\}$ is stationary.
Now, for each $\beta\in S'$, we have $h(\beta)=\sup(\otp(f_\beta(i^*)\cap D))=\sup(\otp((\gamma+1)\cap D))=\sup(\tau+1)=\tau$.\footnote{Note that in this case, we did not need to assume that $\nu$ is a successor cardinal.}

\medskip

\underline{Case 2.} Suppose that $T_2$ is nonstationary.
Define $f:\lambda^+\rightarrow\lambda_{i^*}$ by letting $f(\delta):=f_\delta(i^*)$ for all $\delta<\lambda^+$. 
As $T_2$ is nonstationary, the set $T_4$ of all $\alpha\in T_1$ for which $\delta\mapsto f(\delta)$ is injective over some stationary $S_\alpha\s S\cap\alpha$, is stationary.
Fix $\zeta\le\lambda_{i^*}$ and some stationary subset $T_5\s T_4$ such that, for all $\alpha\in T_5$, $\sup(f[S_\alpha])=\zeta$.\footnote{Indeed, this means that in this case, back at the beginning, we could have chosen $\varepsilon$ to be $\zeta$.}
Let $\psi_\zeta$ be given by Lemma~\ref{localtoglobal}, and then set $\varphi:=(\psi_\zeta\circ f)\restriction S$.
Then $\varphi$ is a function from $S$ to $\nu$ with the property that, for all $\alpha\in T_5$,
there exists a stationary $s_\alpha\s S\cap\alpha$ on which $\varphi$ is strictly increasing.

\medskip

\underline{Case 2.1.}  Suppose that $2^\nu\le\lambda$.
For each $g\in{}^\nu\theta$, we attach a function $h_g:S\rightarrow\theta$ by letting $h_g:=g\circ\varphi$.
We claim that for every $\alpha\in T_5$, there is some $g_\alpha \in{}^\nu\theta$ such that, for all $\tau<\theta$,
$h_{g_\alpha}^{-1}\{\tau\}\cap\alpha$ is stationary in $\alpha$.
Indeed, given $\alpha\in T_5$, we fix a stationary $s_\alpha\s S\cap\alpha$ on which $\varphi$ is injective,
then fix a partition $\langle R_\tau\mid\tau<\theta\rangle$ of $s_\alpha$ into stationary sets,
and then pick  $g:\nu\rightarrow\theta$ such that, for all $\tau<\theta$ and $\delta\in R_\tau$, $g(\varphi(\delta))=\tau$.
Evidently, for all $\tau<\theta$, $h_g^{-1}\{\tau\}$ covers the stationary set $R_\tau$.

Now, as $2^\nu\le\lambda$, fix some stationary $T_6\s T_5$ and some $g\in{}^\nu\theta$ such that $g_\alpha=g$ for all $\alpha\in T_6$.
Then $T_6\s\Tr(h_g^{-1}\{\tau\})$ for all $\tau<\theta$.

\medskip

\underline{Case 2.2.} Suppose that $2^\nu>\lambda$, so that $\nu$ is a successor cardinal, say $\nu=\chi^+$.
Let $\langle A_{\xi,\eta}\mid \xi<\nu, \eta<\chi\rangle$ be an Ulam matrix over $\nu$ \cite{ulam1930masstheorie}. That is:
\begin{itemize}
\item for all $\xi<\nu$,  $|\nu\setminus\bigcup_{\eta<\chi}A_{\xi,\eta}|\le\chi$;
\item for all $\eta<\chi$ and $\xi<\xi'<\nu$, $A_{\xi,\eta}\cap A_{\xi',\eta}=\emptyset$.
\end{itemize}

\begin{claim}\label{claim343} Let $\alpha\in T_6$. There exist $\eta<\chi$ and $x\in[\nu]^\nu$ such that, for all $\xi\in x$, $\varphi^{-1}[A_{\xi,\eta}]\cap\alpha$ is stationary in $\alpha$.
\end{claim}
\begin{proof} Suppose not.
Then, for all $\eta<\chi$, the set $x_\eta:=\{ \xi<\nu\mid \varphi^{-1}[A_{\xi,\eta}]\cap\alpha\allowbreak\text{ is stationary in }\alpha\}$ has size $\le\chi$.
So $X:=\bigcup_{\eta<\chi}x_\eta$ has size $\le\chi$, and we may fix $\xi\in\nu\setminus X$.
It follows that, for all $\eta<\chi$, $\varphi^{-1}[A_{\xi,\eta}]\cap\alpha$ is nonstationary in $\alpha$.
Consequently, $\varphi^{-1}[\bigcup_{\eta<\chi}A_{\xi,\eta}]\cap\alpha$ is nonstationary in $\alpha$.
However, $\bigcup_{\eta<\chi}A_{\xi,\eta}$ contains a tail of $\nu$, contradicting the fact that there exists a stationary $s_\alpha\s S\cap\alpha$ on which $\varphi$ is strictly increasing and converging to $\nu$.
\end{proof}

For each $\alpha\in T_6$, fix $\eta_\alpha$ and $x_\alpha$ as in the claim.
As $\mathcal C(\nu,\theta)\le\lambda$,
fix some stationary $T_7\s T_6$ along with $\eta<\chi$ and  $x\s\nu$ of order-type $\theta$
such that $\eta_\alpha=\eta$ and $x\s \acc^+(x_\alpha)$ for all $\alpha\in T_7$.
Let $h:S\rightarrow\theta$ be any function satisfying $h(\delta):=\sup(\otp(x\cap\xi))$ whenever $\varphi(\delta)\in A_{\xi,\eta}$.
We claim that $T_7\s\Tr( h^{-1}\{\tau\})$ for all $\tau<\theta$.
To see this, let $\alpha\in T_7$ and $\tau<\theta$ be arbitrary.
Let $\xi'$ denote the unique element of $x$ such that $\otp(x\cap\xi')=\tau$.
Put $\xi:=\min(x_\alpha\setminus(\xi'+1))$. As $x\s\acc^+(x_\alpha)$, we know that $[\xi',\xi)\cap x=\{\xi'\}$, so that $\otp(x\cap\xi)=\otp(x\cap(\xi'+1))=\tau+1$.
As $\eta_\alpha=\eta$ and $\xi\in x_\alpha$, the set $S':=\varphi^{-1}[A_{\xi,\eta}]\cap\alpha$ is a stationary subset of $S\cap\alpha$.
Now, for each $\delta\in S'$, we have $\varphi(\delta)\in A_{\xi,\eta}$, meaning that $h(\delta)=\sup(\otp(x\cap\xi))=\sup(\tau+1)=\tau$, as sought.
\end{proof}

\begin{remarks}
\begin{enumerate}
\item It follows from Theorem~\ref{thm35} together with Lemma~\ref{lemmaC} and Fact~\ref{scales} that for every singular cardinal $\lambda$ there exists a partition of $\lambda^+$ into $\lambda$ many reflecting stationary sets.
\item By appealing to a refinement of Fact~\ref{scales}(2), implicitly stated in \cite[Footnote~5]{MR2608402},\footnote{See Lambie-Hanson's answer in \url{https://mathoverflow.net/questions/296225}.}
we infer from Theorem~\ref{thm35} and Lemma~\ref{lemmaC}
that for every singular cardinal $\lambda$, every regular cardinal $\theta$ with $\cf(\lambda)<\theta<\lambda$,
every $S\s\lambda^+$, and every stationary $T\s\Tr(S)\cap E^{\lambda^+}_{\theta^{+3}}$, $\prt(S,\theta^+,T)$ holds.
\end{enumerate}
\end{remarks}

We now prove a variation of Theorem~\ref{thm35} that, compared to its Clause~(i), does not require $\nu$ to be a successor cardinal.
\begin{theorem}\label{thm37} Suppose:
\begin{itemize}
\item $\lambda$ is a singular cardinal;
\item $\vec f = \langle f_\beta\mid\beta<\lambda^+\rangle$ is a scale for $\lambda$;
\item $\chi<\mu<\nu$ are cardinals in $\reg(\lambda)\setminus\{\cf(\lambda)\}$;
\item $S\s E^{\lambda^+}_\mu$ and $T\s E^{\lambda^+}_\nu$ are sets;
\item $G(\vec f)\cap \Tr(S)\cap T$ is stationary.
\end{itemize}
For every  cardinal $\theta\le\nu$ satisfying  $\mathcal C(\nu,\theta)\le\lambda$,\footnote{Note that if $\nu$ is not a successor cardinal, then, by Lemma~\ref{lemmaC}, $\mathcal C(\nu,\theta)<\lambda$ for all $\theta<\nu$.}
$\prt(S,\theta,T)$ holds.
\end{theorem}
\begin{proof} Set $T_0:=G(\vec f)\cap\Tr(S)\cap T$.
\begin{claim} Let $\alpha\in T_0$.
There exist $i_\alpha<\cf(\lambda)$ and $S_\alpha\s E^\alpha_\chi$ such that $\Tr(S_\alpha)\cap S$ is stationary in $\alpha$,
and $\langle f_\gamma(i_\alpha)\mid \gamma\in S_\alpha\rangle$ is strictly increasing.
\end{claim}
\begin{proof}
 As $\alpha$ is good, let us fix a cofinal $A\s\alpha$ and $i<\cf(\alpha)$
such that, for all $\delta<\gamma$ from $A$, $f_\delta<^i f_\gamma$.
Now, for every $\gamma\in\acc^+(A)\cap E^\alpha_\chi$, since $\chi\neq\cf(\lambda)$,
we may fix a cofinal $a_\gamma\s A\cap\gamma$ along with $i_\gamma<\cf(\lambda)$ such that, for all $\delta\in a_\gamma$, $f_\delta<^{i_\gamma}f_\gamma$.
By possibly increasing $i_\gamma$, we may also assume that $f_\gamma<^{i_\gamma}f_{\min(A\setminus(\gamma+1))}$.
Next, for every $\beta\in S\cap\acc(\acc^+(A))$, since $\cf(\beta)=\mu$ and $\mu\neq\cf(\lambda)$, we may find some $i_\beta<\cf(\lambda)$
along with a stationary $S_\beta\s\acc^+(A)\cap E^\beta_\chi$ such that, for all $\gamma\in S_\beta$, $i_\gamma=i_\beta$.
Then, since $\nu\neq\cf(\lambda)$, we may find a stationary $B\s S\cap\acc(\acc^+(A))$ and $i_\alpha<\cf(\lambda)$ such that, for all $\beta\in B$,  $\max\{i_\beta,i\}=i_\alpha$.
Put $S_\alpha:=\bigcup\{ S_\beta\mid\beta\in B\}$. Trivially, $\Tr(S_\alpha)\cap S$ covers the stationary set $B$.
Now, let $\varepsilon<\gamma$ be an arbitrary pair of elements from $S_\alpha$. Find $\beta\le\beta'$ from $B$ such that $\varepsilon\in S_\beta$ and $\gamma\in S_{\beta'}$.
Let $\epsilon:=\min(A\setminus(\varepsilon+1))$. Since $\gamma\in S_{\beta'}\s\acc^+(A)$, we have $\varepsilon<\epsilon<\gamma$.
As $\sup(a_\gamma)=\gamma$, we may also fix $\delta\in a_\gamma$ above $\epsilon$,
so that $\varepsilon<\epsilon<\delta<\gamma$.
By the choice of $i_\varepsilon$ and $i_\gamma$, respectively,  we have $f_\varepsilon<^{i_\varepsilon}f_\epsilon$ and $f_\delta<^{i_\gamma}f_\gamma$.
As $\epsilon,\delta\in A$, we also have $f_\epsilon<^if_\delta$. But $i_\alpha=\max\{i_\beta,i_{\beta'},i\}=\max\{i_{\varepsilon},i_\gamma,i\}$,
so that, altogether, $f_{\varepsilon}<^{i_\alpha}f_\epsilon<^{i_\alpha}f_\delta<^{i_\alpha}f_\gamma$.
Thus, we have established that $\langle f_\gamma(i_\alpha)\mid \gamma\in S_\alpha\rangle$ is strictly increasing.
\end{proof}

For each $\alpha\in T_0$, fix $i_\alpha$ and $S_\alpha$ in the claim. Then find a stationary $T_1\s T_0$ along with $i^*<\cf(\lambda)$ and $\zeta<\lambda$ such that,
for all $\alpha\in T_1$, $i_\alpha=i^*$ and $\langle f_\gamma(i^*)\mid \gamma\in S_\alpha\rangle$ converges to $\zeta$. Define $f:\lambda^+\rightarrow\lambda_{i^*}$ by letting $f(\gamma):=f_\gamma(i^*)$ for all $\gamma<\lambda^+$.
Let $\psi_\zeta$ be given by Lemma~\ref{localtoglobal}, and then put $\varphi:=\psi_\zeta\circ f$. For each $\alpha\in T_1$, pick a club $C_\alpha\s\alpha$ such that $\varphi\restriction(C_\alpha\cap S_\alpha)$ is strictly increasing and converging to $\nu$.
For every $\beta\in S$, fix a strictly increasing function $\pi_\beta:\mu\rightarrow\beta$ whose image is a club in $\beta$. For all $\xi<\nu$ and $\eta<\mu$,  let $A_{\xi,\eta}:=\{\beta\in S\mid \varphi(\pi_\beta(\eta))=\xi\}$.
\begin{claim}\label{claim353} Let $\alpha\in T_1$. There exist $\eta<\mu$ and $x\in[\nu]^\nu$ such that, for all $\xi\in x$, $A_{\xi,\eta}\cap\alpha$ is stationary in $\alpha$.
\end{claim}
\begin{proof} Suppose not.
Then, for all $\eta<\mu$, the set $x_\eta:=\{ \xi<\nu\mid A_{\xi,\eta}\cap\alpha\text{ is}\allowbreak \text{ stationary in }\alpha\}$ has size $<\nu$.
So $X:=\bigcup_{\eta<\mu}x_\eta$ has size $<\nu$, and $\xi:=\sup(X)$ is smaller than $\nu$.
Pick $\gamma\in C_\alpha\cap S_\alpha$ with $\varphi(\gamma)>\xi$.

Next, fix a strictly increasing function $\pi_\alpha:\nu\rightarrow\alpha$ whose image is a club in $\alpha$.
Let $g:=\varphi\circ \pi_\alpha$, so that $g$ is a function from $\nu$ to $\nu$.
Consider  $D:=\{\bar\beta<\nu\mid g[\bar\beta]\s\bar\beta\}$ which is a club in $\nu$,
and  $$B:=\Tr(S_\alpha)\cap S\cap\acc(C_\alpha\setminus\gamma)\cap\acc(\pi_\alpha[D])$$
which is a stationary subset of $\alpha$. Let $\beta\in B$ be arbitrary.
We have that $S_\alpha\cap\beta$ is stationary in $\beta$ and $\im(\pi_\beta)\cap (C_\alpha\setminus\gamma)\cap\pi_\alpha[D]$ is a club in $\beta$, and hence we may find $\eta<\mu$ such that $\pi_\beta(\eta)\in S_\alpha\cap C_\alpha\cap\im(\pi_\alpha)\setminus(\gamma+1)$.
As  $\pi_\beta(\eta)>\gamma$ is a pair of ordinals of $C_\alpha\cap S_\alpha$, we infer that
 $\varphi(\pi_\beta(\eta))>\varphi(\gamma)>\xi$.
In addition, $\pi_\beta(\eta)\in\im(\pi_\alpha)\cap\beta$ and $\beta\in\pi_\alpha[D]$, so that $g(\pi_\alpha^{-1}(\pi_\beta(\eta)))<\pi_\alpha^{-1}(\beta)$.

Thus, we have established that for every $\beta\in B$, there exist $\eta_\beta<\mu$ such that $\xi<\varphi(\pi_{\beta}(\eta_\beta))<\pi_\alpha^{-1}(\beta)$.
As $B$ is stationary in $\alpha$ and $\cf(\alpha)=\nu>\mu$,
we may fix a stationary $B'\s B$ on which the function $\beta\mapsto\eta_\beta$ is constant with value, say, $\eta^*$. So $\bar\beta\mapsto \varphi(\pi_{\pi_\alpha(\bar\beta)}(\eta^*))$ is regressive over $\pi_\alpha^{-1}[B']$, and hence we may find a stationary $B''\s B'$  on which $\beta\mapsto\varphi(\pi_\beta(\eta^*))$ is constant with value, say, $\xi^*$. Then $A_{\xi^*,\eta^*}\cap\alpha$ covers the stationary set $B''$, contradicting the fact that $\xi^*>\xi=\sup(X)$.
\end{proof}
For each $\alpha\in T_1$, fix $\eta_\alpha$ and $x_\alpha$ as in the claim.
As $\mathcal C(\nu,\theta)\le\lambda$,
fix some stationary $T_2\s T_1$ along with $\eta<\chi$ and  $x\s\nu$ of order-type $\theta$
such that $\eta_\alpha=\eta$ and $x\s \acc^+(x_\alpha)$ for all $\alpha\in T_2$.
Let $h:S\rightarrow\theta$ be any function satisfying $h(\delta):=\sup(\otp(x\cap\xi))$ whenever $\delta\in A_{\xi,\eta}$.
We claim that $T_2\s\Tr( h^{-1}\{\tau\})$ for all $\tau<\theta$.
To see this, let $\alpha\in T_2$ and $\tau<\theta$ be arbitrary.
Let $\xi'$ denote the unique element of $x$ such that $\otp(x\cap\xi')=\tau$.
Put $\xi:=\min(x_\alpha\setminus(\xi'+1))$. As $x\s\acc^+(x_\alpha)$, we know that $[\xi',\xi)\cap x=\{\xi'\}$, so that $\otp(x\cap\xi)=\otp(x\cap(\xi'+1))=\tau+1$.
As $\eta_\alpha=\eta$ and $\xi\in x_\alpha$, the set $S':=A_{\xi,\eta}\cap\alpha$ is a stationary subset of $S\cap\alpha$.
Now, for each $\delta\in S'$, we have $\delta \in A_{\xi,\eta}$, meaning that $h(\delta)=\sup(\otp(x\cap\xi))=\sup(\tau+1)=\tau$, as sought.
\end{proof}

Next, we address the case that $\lambda$ is regular.

\begin{theorem}\label{thm38}
Suppose that $\mu<\theta<\lambda$ are infinite regular cardinals, and $T\s E^{\lambda^+}_\theta$ is stationary. Then $\prt(E^{\lambda^+}_\mu,\theta,T)$ holds.
\end{theorem}
\begin{proof}
By \cite[Lemma~4.4]{Sh:351}, $E^{\lambda^+}_{<\lambda}$ is the union of $\lambda$ many sets, each of which carries a partial square.
That is, there exists a sequence $\langle \Gamma_j\mid i<\lambda\rangle$ such that:
\begin{itemize}
\item $\bigcup_{i<\lambda}\Gamma_j=E^{\lambda^+}_{<\lambda}$;
\item for each $j<\lambda$, there is a sequence $\langle C^j_\alpha\mid\alpha\in\Gamma_j\rangle$ such that, for every limit ordinal $\alpha\in\Gamma_j$, $C^j_\alpha$ is a club in $\alpha$ of order-type $<\lambda$,
and for each $\bar\alpha\in\acc(C^j_\alpha)$, we have $\bar\alpha\in\Gamma_j$ and $C^j_{\bar\alpha}=C^j_\alpha\cap\bar\alpha$.
\end{itemize}

Fix $j<\lambda$ such that $T\cap E^{\lambda^+}_\theta\cap\Gamma_j$ is stationary.
Then find $\varepsilon<\lambda$ and a stationary $T_0\s T\cap E^{\lambda^+}_\theta\cap\Gamma_j$ such that, for all $\alpha\in T_0$, $\otp(C_\alpha^j)=\varepsilon$.
By \cite[Lemma~3.1]{paper29}, we may fix a function $\Phi:\mathcal P(\lambda^+)\rightarrow\mathcal P(\lambda^+)$ satisfying that for every $\alpha\in\acc(\lambda^+)$ and every club $x$ in $\alpha$:
\begin{itemize}
\item $\Phi(x)$ is a club in $\alpha$;
\item $\acc(\Phi(x))\s\acc(x)$;
\item if $\bar\alpha\in\acc(\Phi(x))$, then $\Phi(x)\cap\bar\alpha=\Phi(x\cap\bar\alpha)$;
\item if $\otp(x)=\varepsilon$, then $\otp(\Phi(x))=\theta$.
\end{itemize}
For each $\alpha\in \Gamma_j$, let $C_\alpha:=\Phi(C^j_\alpha)$.
Fix $g:\theta\rightarrow\theta$ such that, for all $\tau<\theta$, $E^\theta_\mu\cap g^{-1}\{\tau\}$ is stationary in $\theta$.
Define $h:\lambda^+\rightarrow\theta$ as follows:
$$h(\delta):=\begin{cases}
    g(\otp(C_\delta)),&\text{if }\delta\in\Gamma_j\ \&\ \otp(C_\delta)<\theta;\\
0,&\text{otherwise}.
\end{cases}$$

Now, let $\alpha\in T_0$ and $\tau<\theta$ be arbitrary. As $C_\alpha\cap\delta=C_\delta$ for all $\delta\in\acc(C_\alpha)$, we get that $\langle \otp(C_\delta)\mid \alpha\in\acc(C_\alpha)\rangle$ is a club in $\theta$,
and hence $\{\delta\in E^{\lambda^+}_\mu\cap \acc(C_\alpha)\mid g(\otp(C_\delta))=\tau\}$ is stationary in $\alpha$.
\end{proof}

\begin{remark}\label{remark39}
The proof of the preceding makes clear that if $\mu<\lambda$ are infinite regular cardinals, $\square_\lambda$ holds, and $T\s E^{\lambda^+}_\lambda$  is stationary, then $\prt(E^{\lambda^+}_\mu,\lambda,T)$ holds.
\end{remark}

We are now ready to derive Theorem~A.

\begin{proof}[Proof of Theorem~A]
\begin{enumerate}
\item Suppose that $\lambda$ is inaccessible, so that $\lambda=\aleph_\lambda$.
Trivially, $\langle E^{\lambda^+}_\mu\mid \mu\in\reg(\lambda)\rangle$ witnesses $\prt(E^{\lambda^+}_{<\lambda},\lambda,E^{\lambda^+}_\lambda)$.
Likewise, for cofinally many $\theta<\lambda$ (e.g., $\theta$ singular with $\theta=\aleph_\theta$), $\langle E^{\lambda}_\mu\mid \mu\in\reg(\theta)\rangle$ witnesses $\prt(E^\lambda_{<\theta},\theta,\lambda)$.

\item By Theorem~\ref{thm38}.
\item By Fact~\ref{scales}(\ref{c1}), we may let $\vec f=\langle f_\beta\mid\beta<\lambda^+\rangle$ be some scale for $\lambda$.
Let $\nu:=\theta$ so that $\nu$ is a regular cardinal $\neq\cf(\lambda)$. Let $S:=E^{\lambda^+}_\mu$, and $T:=E^{\lambda^+}_\nu$, so that $\Tr(S)\supseteq T$.
By Fact~\ref{scales}(\ref{c2}), $G(\vec f)\cap E^{\lambda^+}_{\nu}$ is stationary.
So, by Theorem~\ref{thm35}, $\prt(S,\theta,T)$ holds.

\item Let $\vec f$ be some scale for $\lambda$.
Let $\nu:=\min(\{\theta^{+2},\theta^{+3}\}\setminus\{\cf(\lambda)\})$, so that $\nu$ is a successor cardinal $\neq\cf(\lambda)$.
By Lemma~\ref{lemmaC} and as $\theta<\cf(\theta)^+<\nu$, we have $\mathcal C(\nu,\theta)=\nu<\lambda$.
Let $S:=E^{\lambda^+}_\mu$ and $T:=E^{\lambda^+}_\nu$, so that $\Tr(S)\supseteq T$.
Then $G(\vec f)\cap E^{\lambda^+}_{\nu}$ is stationary,
and so, by Theorem~\ref{thm35}, $\prt(S,\theta,T)$ holds.

\end{enumerate}
\end{proof}

We conclude this section by establishing a corollary that was promised at the end of the previous section.
\begin{cor}\label{cor34} For every singular cardinal  $\lambda$ and every $\theta<\lambda$,
 there exists a partition $\langle S_i\mid i<\theta\rangle$ of $\lambda^+$ such that $\sup\{\nu<\lambda\mid E^{\lambda^+}_\nu\cap\bigcap_{i<\theta}\Tr(S_i)\text{ is stationary}\}=\lambda$.
\end{cor}
\begin{proof} Let $A$ be a cofinal subset of $\lambda$ such that each $\mu\in A$ is a  cardinal satisfying $\mu>\max\{\theta,\cf(\lambda)\}$. For each $\mu\in A$, fix a sequence $\langle S_i^\mu\mid i<\theta\rangle$ witnessing  $\prt(E^{\lambda^+}_\mu,\theta,E^{\lambda^+}_{\mu^{++}})$. Then $\langle \lambda^+\setminus\bigcup_{i=1}^\theta\bigcup_{\mu\in A}S^\mu_i\rangle^{\smallfrown}\langle \bigcup_{\mu\in A}S^\mu_i\mid 1\le i<\theta\rangle$ is
 a partition of $\lambda^+$ as sought.
\end{proof}

\section{Theorem~B}
We now introduce a weak consequence of the principle $\snr(\kappa,\nu)$ from \cite{CDSh:571}:

\begin{definition} $\snr^-(\kappa,\nu,T)$ asserts that for every stationary $T_0\s T\cap E^\kappa_\nu$, there exists a function $\varphi:\kappa\rightarrow\nu$ such that, for stationarily many $\alpha\in T_0$,
for some club $c$ in $\alpha$, $\varphi\restriction c$ is strictly increasing.
\end{definition}

The relationship between $\snr^-(\ldots)$ and $\prt(\ldots)$ includes the following.

\begin{theorem}\label{thm42} Suppose:
\begin{itemize}
\item $\nu<\kappa$ are regular uncountable cardinals;
\item $S$ is subset of $\kappa$;
\item  $T\s\Tr(S)\cap E^\kappa_\nu$ is stationary;
\item $\snr^-(\kappa,\nu,  T)$ holds;
\item $\theta\le\nu$ is a cardinal satisfying $\mathcal C(\nu,\theta)<\kappa$.
\end{itemize}
If any of the following holds true:
\begin{enumerate}
\item $\nu$ is a successor cardinal;
\item $S\s E^\kappa_\mu$ for some regular uncountable $\mu<\kappa$;
\end{enumerate}
then $\prt(S,\theta,T)$ holds.
\end{theorem}
\begin{proof}
Fix a function $\varphi:\kappa\rightarrow\nu$, a stationary $T_0\s T$, and a sequence $\vec c=\langle c_\alpha\mid\alpha\in T_0\rangle$ such that, for all $\alpha\in T_0$,
$c_\alpha$ is a club in $\alpha$ (of order-type $\nu$) on which $\varphi$ is strictly increasing.

(1) The proof is similar to that of Theorem~\ref{thm35}, so we only give a sketch.
Suppose that $\nu=\chi^+$ is a successor cardinal.
Let $\langle A_{\xi,\eta}\mid \xi<\nu, \eta<\chi\rangle$ be an Ulam matrix over $\nu$.
For every $\alpha\in T_0$,  $S\cap c_\alpha$ is a stationary subset of $S\cap\alpha$ on which $\varphi$ is injective.
Consequently, and as made by clear by the proof of Claim~\ref{claim343}, there are  $\eta_\alpha<\chi$ and $x_\alpha\in[\nu]^\nu$  such that, for all $\xi\in x_\alpha$, $\varphi^{-1}[A_{\xi,\eta_\alpha}]\cap S\cap \alpha$ is stationary in $\alpha$.
As $\mathcal C(\nu,\theta)<\kappa$,
fix some stationary $T_1\s T_0$ along with $\eta<\chi$ and  $x\s\nu$ of order-type $\theta$
such that $\eta_\alpha=\eta$ and $x\s \acc^+(x_\alpha)$ for all $\alpha\in T_1$.
Let $h:S\rightarrow\theta$ be any function satisfying $h(\delta):=\sup(\otp(x\cap\xi))$ whenever $\varphi(\delta)\in A_{\xi,\eta}$.
Then $T_1\s\Tr(h^{-1}\{\tau\})$ for all $\tau<\theta$.

(2) The proof is similar to that of Theorem~\ref{thm37}.
 For every $\beta\in S$, fix a strictly increasing function $\pi_\beta:\mu\rightarrow\beta$ whose image is a club in $\beta$.
For all $\xi<\nu$ and $\eta<\mu$,  let $A_{\xi,\nu}:=\{\beta\in S\mid \varphi(\pi_\beta(\eta))=\xi\}$.
As made clear by the proof of Claim~\ref{claim353}, for every $\alpha\in T_0$,
there exist $\eta_\alpha<\mu$ and $x_\alpha\in[\nu]^\nu$ such that, for all $\xi\in x_\alpha$, $A_{\xi,\eta_\alpha}\cap\alpha$ is stationary in $\alpha$.\footnote{Note that if $\vec c$ is \emph{coherent}
in the sense that $|\{ c_\alpha\cap\beta\mid \alpha\in T_0, \beta\in\acc(c_\alpha)\}|\le 1$ for all $\beta<\kappa$,
then we can also handle the case $\mu=\aleph_0$. This complements the result mentioned in Remark~\ref{remark39}.}
As $\mathcal C(\nu,\theta)<\kappa$,
fix some stationary $T_1\s T_0$ along with $\eta<\mu$ and  $x\s\nu$ of order-type $\theta$
such that $\eta_\alpha=\eta$ and $x\s \acc^+(x_\alpha)$ for all $\alpha\in T_1$.
Let $h:S\rightarrow\theta$ be any function satisfying $h(\delta):=\sup(\otp(x\cap\xi))$ whenever $\delta\in A_{\xi,\eta}$.
The verification that $h$ witnesses $\prt(S,\theta,T)$ is by now routine.\qedhere
\end{proof}

\begin{theorem}\label{thm43} Suppose:
\begin{itemize}
\item $\nu<\kappa$ are infinite regular cardinals;
\item $S$ is subset of $\kappa$;
\item  $T\s\Tr(S)\cap E^\kappa_\nu$ is stationary;
\item $\snr^-(\kappa,\nu,  T)$ holds;
\item $2^\nu<\kappa$.
\end{itemize}

Then $\prp(S,\nu,T)$ holds.
\end{theorem}
\begin{proof}
Suppose $\theta\le\nu$, $\vec S=\langle S_i\mid i<\theta\rangle$ is a sequence of stationary subsets of $S$,
and $T_0\s T\cap\bigcap_{i<\theta}\Tr(S_i)$ is stationary.
By $\snr^-(\kappa,\nu,T)$, fix a function $\varphi:\kappa\rightarrow\nu$, a stationary $T_1\s T_0$ and a sequence $\langle c_\alpha\mid \alpha\in T_1\rangle$ such that, for every $\alpha\in T_1$, $c_\alpha$ is a club in $\alpha$ (of order-type $\nu$) on which $\varphi$ is injective.

\begin{claim}
Let $\alpha\in T_1$.  Then there exists a function $h:\nu\rightarrow\theta$ such that $\im(h)\in[\theta]^\theta$ and, for all $i\in\im(h)$, $\{\delta\in S_i\cap\alpha\mid h(\varphi(\delta))=i\}$ is stationary in $\alpha$.
\end{claim}
\begin{proof} Let $\pi:c_\alpha\rightarrow\nu$ denote the unique order-preserving bijection. As $\alpha\in T_1\s\bigcap_{i<\theta}\Tr(S_i)$,
we know that $\langle \pi[S_i]\mid i<\theta\rangle$ is a sequence of stationary subsets of $\nu$, so by Lemma~\ref{lemma13}, we fix a sequence $\langle S_i'\mid i\in I\rangle$ of pairwise disjoint sets such that:
\begin{itemize}
\item $I$ is a cofinal subset of $\theta$;
\item for each $i\in I$, $S_i'\s \pi[S_i]$ is stationary.
\end{itemize}
Now, as $\varphi\circ\pi^{-1}$ is injective, it easy to find $h:\nu\rightarrow I$ satisfying that, for all $i\in I$ and $\bar\delta\in S_i'$, $h(\varphi(\pi^{-1}(\bar\delta)))=i$. Clearly, any such $h$ is as sought.
\end{proof}
For each $\alpha\in T_1$, fix a function $h_\alpha$ as in the claim.
Then, as $2^\nu<\kappa$, we may find a stationary $T_2\s T_1$ and some $h:\nu\rightarrow\theta$ such that $h_\alpha=h$ for all $\alpha\in T_2$.
Let $I:=\im(h)$. For each $i\in I$, let $S_i':=\{\delta\in S_i\mid h(\varphi(\delta))=i\}$. Clearly, $\langle S'_i\mid i\in I\rangle$ is a sequence as sought.
\end{proof}

\begin{prop} Suppose that $\nu<\lambda=\lambda^\nu<\kappa\le 2^\lambda$ are infinite cardinals with $\nu,\kappa$ regular.
Then:
 \begin{enumerate}
 \item $\snr^-(\kappa,\nu,E^\kappa_\nu)$ holds;
\item  $\prp(\kappa,\nu,E^\kappa_\nu)$ holds.
 \end{enumerate}
\end{prop}
\begin{proof} (1) By a standard application of the Engelking-Kar{\l}owicz theorem.

(2) By Clause~(1) and Theorem~\ref{thm43}, noticing that $2^\nu\le\lambda<\kappa$.
\end{proof}

The next scenario arises naturally when one tries to relax the hypothesis ``$S\s E^{\lambda^+}_\mu$'' of Theorem~\ref{thm37} into ``$S\s \lambda^+$''.

\begin{prop} Suppose that $\nu<\kappa$ are regular uncountable cardinals, $T\s E^\kappa_\nu$,
and there exists a function $f:\kappa\rightarrow\nu$ such that $T\cap\bigcup_{i<\nu}\Tr(f^{-1}\{i\})$ is nonstationary.
Then $\snr^-(\kappa,\nu,T)$ holds.
\end{prop}
\begin{proof}  Fix such a function $f:\kappa\rightarrow\nu$. Denote $S_i:=f^{-1}\{i\}$. Let $\alpha\in T\setminus\bigcup_{i<\nu}\Tr(S_i)$ be arbitrary.
For each $i<\nu$, fix a club $c_\alpha^i$ in $\alpha$ disjoint from $S_i$. Let $\pi:\nu\rightarrow\alpha$ denote the inverse collapse of some club in $\alpha$,
and let $c_\alpha:=\pi[\diagonal_{i<\nu}\pi^{-1}[c_\alpha^i]]$.
Then $c_\alpha$ is a club in $\alpha$, and, for all $\beta\in c_\alpha$ and $i<\pi^{-1}(\beta)$, we have $\beta\in c^i_\alpha$ so that $f(\beta)\neq i$.
Consequently, $f(\beta)\ge\pi^{-1}(\beta)$ for all $\beta\in c_\alpha$.
Therefore, there exists a club $c_\alpha'\s c_\alpha$ on which $f$ is strictly increasing.
\end{proof}

\begin{prop}\label{prop45} Suppose $\vec f=\langle f_\beta\mid\beta<\lambda^+\rangle$ is a scale for a singular cardinal $\lambda$, and $\nu\in\reg(\lambda)\setminus\{\aleph_0,\cf(\lambda)\}$.
Then $\snr^-(\lambda^+,\nu,E^{\lambda^+}_\nu\cap V(\vec f))$ holds.\footnote{Recall Definition~\ref{def22}.}
\end{prop}
\begin{proof}
Let $T_0$ be an arbitrary stationary subset of $E^{\lambda^+}_\nu\cap V(\vec f)$.
For each $\alpha\in T_0$, fix a club $c_\alpha\s\alpha$ and some $i_\alpha<\cf(\lambda)$ such that for any pair $\delta<\gamma$ of ordinals of $c_\alpha$,
$f_\delta<^{i_\alpha} f_\gamma$. Fix a stationary $T_1\s T_0$, and ordinals $i<\cf(\lambda)$, $\zeta<\lambda$ such that, for all $\alpha\in T_1$,
$\langle f_\delta(i)\mid \delta\in c_\alpha\rangle$ is strictly increasing and converging to $\zeta$.
Let $\psi_\zeta$ be given by Lemma~\ref{localtoglobal}.
Define $\varphi:\lambda^+\rightarrow\nu$ by letting $\varphi(\delta):=\psi_\zeta(f_\delta(i))$ for all $\delta<\lambda^+$.
Clearly, for every $\alpha\in T_1$, there exists a club $c_\alpha'\s c_\alpha$ on which $\varphi$ is strictly increasing.
\end{proof}

\begin{cor}\label{cor45} Suppose that $\nu,\lambda$ are cardinals with $\aleph_0<\cf(\nu)=\nu<\cf(\lambda)$.
Then:
\begin{enumerate}
\item $\snr^-(\lambda^+,\nu,E^{\lambda^+}_\nu)$ holds;
\item If $2^\nu\le\lambda$, then  $\prp(\lambda^+,\nu,E^{\lambda^+}_\nu)$ holds.
\end{enumerate}
\end{cor}
\begin{proof} (1) If $\lambda$ is singular, then by Fact~\ref{scales}(1), let us fix a scale $\vec f$ for $\lambda$.
As $\nu<\cf(\lambda)$, we have $E^{\lambda^+}_\nu\s V(\vec f)$. Now, appeal to Proposition~\ref{prop45}.

Next, suppose that $\lambda$ is regular.
Let $T_0$ be an arbitrary stationary subset of $E^{\lambda^+}_{\nu}$. As made clear by the proof of Theorem~\ref{thm38}, there exists a sequence $\langle C_\alpha\mid\alpha\in\Gamma\rangle$ such that:
\begin{itemize}
\item  $\Gamma\s\acc(\lambda^+)$;
\item for all $\alpha\in\Gamma$, $C_\alpha$ is a club in $\alpha$;
\item for all $\alpha\in\Gamma$ and $\bar\alpha\in\acc(C_\alpha)$,  $\bar\alpha\in\Gamma$ and $C_{\bar\alpha}=C_\alpha\cap\bar\alpha$;
\item $T_1:=\{\alpha\in\Gamma\cap T_0\mid \otp(C_\alpha)=\nu\}$ is stationary.
\end{itemize}
Now, define $\varphi:\kappa\rightarrow\nu$ by letting $\varphi(\alpha):=\otp(C_\alpha)$ whenever $\alpha\in\Gamma$ and $\otp(C_\alpha)<\nu$; otherwise, let $\varphi(\alpha):=0$.
Then $\langle C_\alpha\mid\alpha\in T_1\rangle$ witnesses that $\varphi$ is as sought.

(2) By Clause~(1) and Theorem~\ref{thm43}.
\end{proof}

We are now ready to derive Theorem~B.
\begin{proof}[Proof of Theorem~B]
If $\lambda$ is a singular cardinal admitting a very good scale, then by Proposition~\ref{prop45}, $\snr^-(\lambda^+,\nu,E^{\lambda^+}_\nu)$ holds.
If $\lambda$ is a regular cardinal, then by Corollary~\ref{cor45}(1), $\snr^-(\lambda^+,\nu,E^{\lambda^+}_\nu)$ holds.
Now, appeal to Theorem~\ref{thm43}.
\end{proof}

\section*{Acknowledgments} 
The main results of this paper were presented by the first author at the \emph{$7^{th}$ European Set Theory Conference},
Vienna, July 2019, and by second author at the \emph{Arctic Set Theory Workshop 4}, Kilpisj\"arvi, January 2019. We thank the organizers and the participants for their feedback.

\end{document}